\documentclass{article}
\usepackage{amsmath}
\usepackage{amssymb}
\usepackage{amsfonts}

\newtheorem{theorem}{Theorem}[section]

\newtheorem{criterion}{Criterion}[section]
\newtheorem{definition}[theorem]{Definition}
\newtheorem{example}{Example}

\newtheorem{proposition}{Proposition}
[section]

\newenvironment{proof}[1][Proof]{\noindent\textbf{#1.} }{\ \rule{0.5em}{0.5em}}

\input{tcilatex}
\begin{document}

\title{{\Large Automorphic Equivalence in the Classical Varieties of Linear
Algebras.}}
\author{{\Large A.Tsurkov} \\
Institute of Mathematics and Statistics.\\
University S\~{a}o Paulo. \\
Rua do Mat\~{a}o, 1010 \\
Cidade Universit\'{a}ria \\
S\~{a}o Paulo - SP - Brazil - CEP 05508-090 \\
arkady.tsurkov@gmail.com}
\maketitle

\begin{abstract}
This research is a continuation of the \cite{Tsurkov}. In this paper we
consider some classical varieties of linear algebras over the field $k$ such
that $char\left( k\right) =0$. We study the relation between the geometric
equivalence and automorphic equivalence of the algebras of these varieties.

If we denote by $\Theta $ one of these varieties, then $\Theta ^{0}$ is a
category of the finite generated free algebras of the variety $\Theta $. In
this paper we calculate for the considered varieties the quotient group $%
\mathfrak{A/Y}$, where $\mathfrak{A}$ is a group of the all automorphisms of
the category $\Theta ^{0}$ and $\mathfrak{Y}$ is a subgroup of the all inner
automorphisms of this category. The quotient group $\mathfrak{A/Y}$ measures
difference between the geometric equivalence and automorphic equivalence of
algebras from the variety $\Theta $. The results of this paper and of the 
\cite{Tsurkov} are summarized in the table in the end of the Section \ref%
{varieties}.

We can see from this table that in the all considered varieties of the
linear algebras the group $\mathfrak{A/Y}$ is generated by cosets which are
presented by no more than two kinds of the strongly stable automorphisms of
the category $\Theta ^{0}$. One kind of automorphisms is connected to the
changing of the multiplication by scalar and second one is connected to the
changing of the multiplication of the elements of the algebras. In the
Section \ref{examples} we present some examples of the pairs of linear
algebras such that the considered automorphism provides the automorphic
equivalence of these algebras but these algebras are not geometrically
equivalent. These examples are presented for the all considered above
varieties of algebras and for both these kinds of the strongly stable
automorphisms, when they exist in the group $\mathfrak{A/Y}$.
\end{abstract}

\section{Introduction.\label{intro}}

\setcounter{equation}{0}

This is a paper from universal algebraic geometry. All definitions of the
basic notions of the universal algebraic geometry can be found, for example,
in \cite{PlotkinVarCat}, \cite{PlotkinNotions} and \cite{PlotkinSame}.

In universal algebraic geometry we consider some variety $\Theta $ of
one-sorted algebras of the signature $\Omega $. We denote by $X_{0}=\left\{
x_{1},x_{2},\ldots ,x_{n},\ldots \right\} $ a countable set of symbols, and
by $\mathfrak{F}\left( X_{0}\right) $ the set of all finite subsets of $%
X_{0} $. We will consider the category $\Theta ^{0}$, whose objects are all
free algebras $F\left( X\right) $ of the variety $\Theta $ generated by
finite subsets $X\in \mathfrak{F}\left( X_{0}\right) $. Morphisms of the
category $\Theta ^{0}$ are homomorphisms of free algebras.

We denote some time $F\left( X\right) =F\left( x_{1},x_{2},\ldots
,x_{n}\right) $ if $X=\left\{ x_{1},x_{2},\ldots ,x_{n}\right\} $ and even $%
F\left( X\right) =F\left( x\right) $ if $X$ has only one element.

We consider a "set of equations" $T\subset F\times F$, where $F\in \mathrm{Ob%
}\Theta ^{0}$, and we "resolve" these equations in $\mathrm{Hom}\left(
F,H\right) $, where $H\in \Theta $. The set $\mathrm{Hom}\left( F,H\right) $
serves as an "affine space over the algebra $H$". We denote by $%
T_{H}^{\prime }$ the set $\left\{ \mu \in \mathrm{Hom}\left( F,H\right) \mid
T\subset \ker \mu \right\} $. This is the set of all solutions of the set of
equations $T$. For every set of "points" $R\subset \mathrm{Hom}\left(
F,H\right) $ we consider a congruence of equations defined by this set: $%
R_{H}^{\prime }=\bigcap\limits_{\mu \in R}\ker \mu $. For every set of
equations $T$ we consider its algebraic closure $T_{H}^{\prime \prime }$ in
respect to the algebra $H$. A set $T\subset F\times F$ is called $H$-closed
if $T=T_{H}^{\prime \prime }$. An $H$-closed set is always a congruence.

\begin{definition}
Algebras $H_{1},H_{2}\in \Theta $ are \textbf{geometrically equivalent} if
and only if for every $X\in \mathfrak{F}\left( X_{0}\right) $ and every $%
T\subset F\left( X\right) \times F\left( X\right) $ fulfills $%
T_{H_{1}}^{\prime \prime }=T_{H_{2}}^{\prime \prime }$.
\end{definition}

We denote the family of all $H$-closed congruences in $F$ by $Cl_{H}(F)$. We
can consider the category $C_{\Theta }\left( H\right) $ of the \textit{%
coordinate algebras} connected with the algebra $H\in \Theta $. Objects of
this category are quotient algebras $F\left( X\right) /T$, where $X\in 
\mathfrak{F}\left( X_{0}\right) $, $T\in Cl_{H}(F\left( X\right) )$.
Morphisms of this category are homomorphisms of algebras.

\begin{definition}
\label{automorphic_equivalence}Let $Id\left( H,X\right)
=\bigcap\limits_{\varphi \in \mathrm{Hom}\left( F\left( X\right) ,H\right)
}\ker \varphi $ be the minimal $H$-closed congruence in $\ F\left( X\right) $%
. Algebras $H_{1},H_{2}\in \Theta $ are \textbf{automorphically equivalent}
if and only if there exists a pair $\left( \Phi ,\Psi \right) ,$ where $\Phi
:\Theta ^{0}\rightarrow \Theta ^{0}$ is an automorphism,\textit{\ }$\Psi
:C_{\Theta }\left( H_{1}\right) \rightarrow $\textit{\ }$C_{\Theta }\left(
H_{2}\right) $ is an isomorphism subject to conditions:
\end{definition}

\begin{enumerate}
\item[A.] $\Psi \left( F\left( X\right) /Id\left( H_{1},X\right) \right)
=F\left( Y\right) /Id\left( H_{2},Y\right) $\textit{, where }$\Phi \left(
F\left( X\right) \right) =F\left( Y\right) $\textit{,}

\item[B.] $\Psi \left( F\left( X\right) /T\right) =F\left( Y\right) /%
\widetilde{T}$\textit{, where }$T\in Cl_{H_{1}}(F\left( X\right) )$\textit{, 
}$\widetilde{T}\in Cl_{H_{2}}(F\left( Y\right) )$\textit{,}

\item[C.] $\Psi $\textit{\ takes the natural epimorphism }$\overline{\tau }%
:F\left( X\right) /Id\left( H_{1},X\right) \rightarrow F\left( X\right) /T$%
\textit{\ \ to the natural epimorphism }$\Psi \left( \overline{\tau }\right)
:F\left( Y\right) /Id\left( H_{2},Y\right) \rightarrow F\left( Y\right) /%
\widetilde{T}$,

\item[D.] \textit{for every }$F\left( X_{1}\right) ,F\left( X_{2}\right) \in 
\mathrm{Ob}\Theta ^{0}$\textit{\ and every}%
\begin{equation*}
\nu \in \mathrm{Mor}_{C_{\Theta }\left( H_{1}\right) }\left( F\left(
X_{1}\right) /Id\left( H_{1},X_{1}\right) ,F\left( X_{2}\right) /Id\left(
H_{1},X_{2}\right) \right)
\end{equation*}%
\textit{if the diagram}%
\begin{equation*}
\begin{array}{ccc}
F\left( X_{1}\right) & \underset{\delta _{1}}{\rightarrow } & F\left(
X_{1}\right) /Id\left( H_{1},X_{1}\right) \\ 
\downarrow \mu &  & \nu \downarrow \\ 
F\left( X_{2}\right) & \overset{\delta _{2}}{\rightarrow } & F\left(
X_{2}\right) /Id\left( H_{1},X_{2}\right)%
\end{array}%
\end{equation*}%
\textit{is commutative then the diagram}%
\begin{equation*}
\begin{array}{ccc}
F\left( Y_{1}\right) & \underset{\widetilde{\delta _{1}}}{\rightarrow } & 
F\left( Y_{1}\right) /Id\left( H_{2},Y_{1}\right) \\ 
\downarrow \Phi \left( \mu \right) &  & \Psi \left( \nu \right) \downarrow
\\ 
F\left( Y_{2}\right) & \overset{\widetilde{\delta _{2}}}{\rightarrow } & 
F\left( Y_{2}\right) /Id\left( H_{2},Y_{2}\right)%
\end{array}%
\end{equation*}%
\textit{is also commutative, where }$\mu \in \mathrm{Mor}_{\Theta
^{0}}\left( F\left( X_{1}\right) ,F\left( X_{2}\right) \right) $,\textit{\ }$%
\delta _{i}$\textit{\ and }$\widetilde{\delta _{i}}$\textit{\ are the
natural epimorphisms, }$\Phi \left( F\left( X_{i}\right) \right) =F\left(
Y_{i}\right) $\textit{\ }$i=1,2$\textit{.}
\end{enumerate}

If we will compare the geometric equivalence and the automorphic equivalence
of the one-sorted universal algebras from the some variety $\Theta $, we
must take a countable set of symbols $X_{0}=\left\{ x_{1},x_{2},\ldots
,x_{n},\ldots \right\} $ and consider all free algebras $F\left( X\right) $
of the variety $\Theta $, generated by finitely subsets $X\subset X_{0}$.
These algebras: $\left\{ F\left( X\right) \mid X\subset X_{0},\left\vert
X\right\vert <\infty \right\} $ - will be objects of the category $\Theta
^{0}$. Morphisms of the category $\Theta ^{0}$ will be homomorphisms of
these algebras.

If our variety $\Theta $ is a variety of one-sorted algebras and possesses
the IBN property: for free algebras $F\left( X\right) ,F\left( Y\right) \in
\Theta $ we have $F\left( X\right) \cong F\left( Y\right) $ if and only if $%
\left\vert X\right\vert =\left\vert Y\right\vert $ - then we have \cite[%
Theorem 2]{PlotkinZhitAutCat} the decomposition%
\begin{equation}
\mathfrak{A=YS}.  \label{decomp}
\end{equation}%
of the group $\mathfrak{A}$ of all automorphisms of the category\textit{\ }$%
\Theta ^{0}$. Hear $\mathfrak{Y}$ is a group of all inner automorphisms of
the category\textit{\ }$\Theta ^{0}$ and $\mathfrak{S}$ is a group of all
strongly stable automorphisms of the category\textit{\ }$\Theta ^{0}$. The
definitions of the notions of inner automorphisms and strongly stable
automorphisms can be found, for example, in \cite{PlotkinZhitAutCat}, \cite%
{Tsurkov} and \cite{TsurkovManySortes}. But we will give these definitions
hear.

\begin{definition}
\label{inner}An automorphism $\Upsilon $ of a category $\mathfrak{K}$ is 
\textbf{inner}, if it is isomorphic as a functor to the identity
automorphism of the category $\mathfrak{K}$.
\end{definition}

It means that for every $A\in \mathrm{Ob}\mathfrak{K}$ there exists an
isomorphism $s_{A}^{\Upsilon }:A\rightarrow \Upsilon \left( A\right) $ such
that for every $\psi \in \mathrm{Mor}_{\mathfrak{K}}\left( A,B\right) $ the
diagram%
\begin{equation*}
\begin{array}{ccc}
A & \overrightarrow{s_{A}^{\Upsilon }} & \Upsilon \left( A\right) \\ 
\downarrow \psi &  & \Upsilon \left( \psi \right) \downarrow \\ 
B & \underrightarrow{s_{B}^{\Upsilon }} & \Upsilon \left( B\right)%
\end{array}%
\end{equation*}%
\noindent commutes.

\begin{definition}
\label{str_stab_aut}\textbf{\hspace{-0.08in}. }\textit{An automorphism $\Phi 
$ of the category }$\Theta ^{0}$\textit{\ is called \textbf{strongly stable}
if it satisfies the conditions:}

\begin{enumerate}
\item[StSt1)] $\Phi $\textit{\ preserves all objects of }$\Theta ^{0}$%
\textit{,}

\item[StSt2)] \textit{there exists a system of bijections }$\left\{
s_{F}^{\Phi }:F\rightarrow F\mid F\in \mathrm{Ob}\Theta ^{0}\right\} $%
\textit{\ such that }$\Phi $\textit{\ acts on the morphisms $\psi
:D\rightarrow F$ of }$\Theta ^{0}$\textit{\ by this way: }%
\begin{equation}
\Phi \left( \psi \right) =s_{F}^{\Phi }\psi \left( s_{D}^{\Phi }\right)
^{-1},  \label{biject_action}
\end{equation}

\item[StSt3)] $s_{F}^{\Phi }\mid _{X}=id_{X},$ \textit{\ for every free
algebra} $F=F\left( X\right) $.
\end{enumerate}
\end{definition}

The subgroup $\mathfrak{Y}$ is a normal in $\mathfrak{A}$.

By \cite{PlotkinSame} only strongly stable automorphism $\Phi $ can provide
us automorphic equivalence of algebras which not coincides with geometric
equivalence of algebras. Therefore, in some sense, difference from the
automorphic equivalence to the geometric equivalence is measured by the
quotient group $\mathfrak{A/Y\cong S/S\cap Y}$.

\section{Verbal operations and strongly stable automorphisms.\label%
{operations}}

\setcounter{equation}{0}

In this paper, as in the \cite{Tsurkov} we use the method of verbal
operations for the finding of the strongly stable automorphisms of the
category $\Theta ^{0}$. The explanation of this method there is in \cite%
{PlotkinZhitAutCat}, \cite{TsurkovAutomEquiv} and \cite{TsurkovManySortes}.

We denote the signature of our variety $\Theta $ by $\Omega $, by $m_{\omega
}$ we denote the arity of $\omega $ for every $\omega \in \Omega $. If $%
w=w\left( x_{1},\ldots ,x_{m_{\omega }}\right) \in F\left( x_{1},\ldots
,x_{m_{\omega }}\right) \in \mathrm{Ob}\Theta ^{0}$, then we can define in
every algebra $H\in \Theta $ by using of the this word $w$ the new operation 
$\omega ^{\ast }$: $\omega ^{\ast }\left( h_{1},\ldots ,h_{m_{\omega
}}\right) =w\left( h_{1},\ldots ,h_{m_{\omega }}\right) $ for every $%
h_{1},\ldots ,h_{m_{\omega }}\in H$. This operation we call the \textbf{%
verbal operation} defined on the algebra $H$ by the word $w$. If we have a
system of words $W=\left\{ w_{\omega }\mid \omega \in \Omega \right\} $ such
that $w_{\omega }\in F\left( x_{1},\ldots ,x_{m_{\omega }}\right) $ then we
denote by $H_{W}^{\ast }$ the algebra which coincide with $H$ as a set but
instead the original operations $\left\{ \omega \mid \omega \in \Omega
\right\} $ it has the system of the verbal operations $\left\{ \omega ^{\ast
}\mid \omega \in \Omega \right\} $ defined by words from the system $W$.

We suppose that we have the system of words $W=\left\{ w_{\omega }\mid
\omega \in \Omega \right\} $ satisfies the conditions:

\begin{enumerate}
\item[Op1)] $w_{\omega }\left( x_{1},\ldots ,x_{m_{\omega }}\right) \in
F\left( x_{1},\ldots ,x_{m_{\omega }}\right) \in \mathrm{Ob}\Theta ^{0}$,

\item[Op2)] for every $F=F\left( X\right) \in \mathrm{Ob}\Theta ^{0}$ there
exists an isomorphism $\sigma _{F}:F\rightarrow F_{W}^{\ast }$ such that $%
\sigma _{F}\mid _{X}=id_{X}$.
\end{enumerate}

It is clear isomorphisms $\sigma _{F}$ are defined uniquely by the system of
words $W$.

The set $S=\left\{ \sigma _{F}:F\rightarrow F\mid F\in \mathrm{Ob}\Theta
^{0}\right\} $ is a system of bijections which satisfies the conditions:

\begin{enumerate}
\item[B1)] for every homomorphism $\psi :A\rightarrow B\in \mathrm{Mor}%
\Theta ^{0}$ the mappings $\sigma _{B}\psi \sigma _{A}^{-1}$ and $\sigma
_{B}^{-1}\psi \sigma _{A}$ are homomorphisms;

\item[B2)] $\sigma _{F}\mid _{X}=id_{X}$ for every free algebra $F\in 
\mathrm{Ob}\Theta ^{0}$.
\end{enumerate}

So we can define the strongly stable automorphism\textit{\ }by this system
of bijections. This automorphism preserves all objects of $\Theta ^{0}$ and
acts on morphism of $\Theta ^{0}$ by formula (\ref{biject_action}), where $%
s_{F}^{\Phi }=$ $\sigma _{F}$.

Vice versa if we have a strongly stable automorphism $\Phi $ of the category 
$\Theta ^{0}$ then its system of bijections $S=\left\{ s_{F}^{\Phi
}:F\rightarrow F\mid F\in \mathrm{Ob}\Theta ^{0}\right\} $ defined uniquely.
Really, if $F\in \mathrm{Ob}\Theta ^{0}$ and $f\in F$ then%
\begin{equation}
s_{F}^{\Phi }\left( f\right) =s_{F}^{\Phi }\psi \left( x\right) =\left(
s_{F}^{\Phi }\psi \left( s_{D}^{\Phi }\right) ^{-1}\right) \left( x\right)
=\left( \Phi \left( \psi \right) \right) \left( x\right) ,
\label{autom_bijections}
\end{equation}%
where $D=F\left( x\right) $ - $1$-generated free linear algebra - and $\psi
:D\rightarrow F$ homomorphism such that $\psi \left( x\right) =f$. Obviously
that this system of bijections $S=\left\{ s_{F}^{\Phi }:F\rightarrow F\mid
F\in \mathrm{Ob}\Theta ^{0}\right\} $ fulfills conditions B1) and B2) with $%
\sigma _{F}=s_{F}^{\Phi }$.

If we have a system of bijections $S=\left\{ \sigma _{F}:F\rightarrow F\mid
F\in \mathrm{Ob}\Theta ^{0}\right\} $ which fulfills conditions B1) and B2)
than we can define the system of words $W=\left\{ w_{\omega }\mid \omega \in
\Omega \right\} $ satisfies the conditions Op1) and Op2) by formula%
\begin{equation}
w_{\omega }\left( x_{1},\ldots ,x_{m_{\omega }}\right) =\sigma _{F_{\omega
}}\left( \omega \left( \left( x_{1},\ldots ,x_{m_{\omega }}\right) \right)
\right) \in F_{\omega },  \label{der_veb_opr}
\end{equation}%
where $F_{\omega }=F\left( x_{1},\ldots ,x_{m_{\omega }}\right) $.

By formulas (\ref{autom_bijections}) and (\ref{der_veb_opr}) we can check
that there are

\begin{enumerate}
\item one to one and onto correspondence between strongly stable
automorphisms of the category $\Theta ^{0}$ and systems of bijections
satisfied the conditions B1) and B2)

\item one to one and onto correspondence between systems of bijections
satisfied the conditions B1) and B2) and systems of words satisfied the
conditions Op1) and Op2).
\end{enumerate}

Therefore we can calculate the group $\mathfrak{S}$ if we can find the all
system of words which fulfill conditions Op1) and Op2). For calculation of
the group $\mathfrak{S\cap Y}$ we also have a

\begin{criterion}
\label{inner_stable}The strongly stable automorphism $\Phi $ of the category 
$\Theta ^{0}$ which corresponds to the system of words $W$ is inner if and
only if for every $F\in \mathrm{Ob}\Theta ^{0}$ there exists an isomorphism $%
c_{F}:F\rightarrow F_{W}^{\ast }$ such that $c_{F}\psi =\psi c_{D}$ fulfills
for every $\left( \psi :D\rightarrow F\right) \in \mathrm{Mor}\Theta ^{0}$.
\end{criterion}

\section{Verbal operations in linear algebras.\label%
{operations_in_linear_alg}}

\setcounter{equation}{0}

From now on, the variety $\Theta $ will be some specific variety of the
linear algebras over field $k$ with characteristic $0$. We nether consider
the vanished varieties, i., e., variety defined by identity $x=0$ or variety
defined by identity $x_{1}x_{2}=0$. For linear algebras we can rewrite the
equation $t_{1}=t_{2}$, where $t_{1},t_{2}\in F$, $F\in \mathrm{Ob}\Theta
^{0}$, as $t_{1}-t_{2}=0$. So we can assume that every set of equations $T$
is a subset of $F\in \mathrm{Ob}\Theta ^{0}$ and this set $T\subset F$ we
understand as $\left\{ t=0\mid t\in T\right\} $. Congruences we can consider
as two sided ideals of the algebra $F$. From now on, the word "ideal" means
two sided ideal of the linear algebra.

We consider linear algebras as one-sorted universal algebras, i. e.,
multiplication by scalar we consider as $1$-ary operation for every $\lambda
\in k$: $H\ni h\rightarrow \lambda h\in H$ where $H\in \Theta $. Hence the
signature $\Omega $ of algebras of our variety contains these operations: $0$%
-ary operation $0$; $\left\vert k\right\vert $ $1$-ary operations of
multiplications by scalars; $2$-ary operation $\cdot $ and $2$-ary operation 
$+$. We will finding the system of words $W=\left\{ w_{\omega }\mid \omega
\in \Omega \right\} $ satisfies the conditions Op1) and Op2). We denote the
words corresponding to these operations by $w_{0}$, $w_{\lambda }$ for all $%
\lambda \in k$, $w_{\cdot }$, $w_{+}$. So%
\begin{equation}
W=\left\{ w_{\omega }\mid \omega \in \Omega \right\} =\left\{
w_{0},w_{\lambda }\left( \lambda \in k\right) ,w_{+},w_{\cdot }\right\}
\label{words_list}
\end{equation}%
in our case. From this on we consider only these systems of words.

Some time we denote by $\lambda \ast $ the operation defined by the word $%
w_{\lambda }\left( \lambda \in k\right) $, by $\perp $ the operation defined
by the word $w_{+}$ and by $\times $ the operation defined by the word $%
w_{\cdot }$.

We denote the group of all automorphisms of the field $k$ by $\mathrm{Aut}k$.

We use in our research the familiar fact that every variety of the linear
algebras over field with characteristic $0$ is multi-homogenous. So, for
example, every $F\left( X\right) \in \mathrm{Ob}\Theta ^{0}$ can be
decompose to the direct sum of the linear spaces of elements which are
homogeneous according the sum of degrees of generators from the set $X$: $%
F\left( X\right) =\bigoplus\limits_{i=1}^{\infty }F_{i}$. We also denote the
ideals $\bigoplus\limits_{i=j}^{\infty }F_{i}=F^{j}$. $F_{i}F_{j}\subset
F_{i+j}$ and $F^{i}F^{j}\subset F^{i+j}$ fulfills for every $1\leq
i,j<\infty $.

All our varieties $\Theta $ possess the IBN property, because $\left\vert
X\right\vert =\dim F/F^{2}$ fulfills for all free algebras $F=F\left(
X\right) \in \mathrm{Ob}\Theta ^{0}$. So we have the decomposition (\ref%
{decomp}) for group of all automorphisms of the category $\Theta ^{0}$.

\section{Classical varieties of linear algebras.\label{varieties}}

\setcounter{equation}{0}

In this Section we consider as the variety $\Theta $ the varieties of the
all commutative algebras, of the all power associative algebras, i., e., the
variety of linear algebras defined by identities%
\begin{equation*}
x\left( x^{2}\right) =\left( x^{2}\right) x,
\end{equation*}%
\begin{equation}
x\left( x\left( x^{2}\right) \right) =x\left( \left( x^{2}\right) x\right)
=\left( x\left( x^{2}\right) \right) x=\left( \left( x^{2}\right) x\right)
x=\left( x^{2}\right) \left( x^{2}\right)  \label{p_ass_id}
\end{equation}%
and so on, of the all alternative algebras, of the all Jordan algebras and
arbitrary subvariety defined by identities with coefficients from $%
\mathbb{Z}
$ of the variety of the all anticommutative algebras.

For the calculating of the group $\mathfrak{S}$ we consider an arbitrary
strongly stable automorphism $\Phi $ of the category $\Theta ^{0}$\ and we
will find for the all possible forms of the system of words $W$ which
corresponds to the automorphism $\Phi $.

For the all considered varieties $F\left( \varnothing \right) =\left\{
0\right\} $, so $w_{0}=0$.

The crucial point is the finding of the words $w_{\lambda }\left( x\right)
\in F\left( x\right) $, where $\lambda \in k$. The system of words $W$ must
fulfills conditions Op1) and Op2). By condition Op2) all axioms of the
variety $\Theta $ must hold in the $F_{W}^{\ast }$ for every $F\in \mathrm{Ob%
}\Theta ^{0}$. For every $\lambda \in k^{\ast }$ must holds 
\begin{equation}
w_{\lambda ^{-1}}\left( w_{\lambda }\left( x\right) \right) =w_{\lambda
}\left( w_{\lambda ^{-1}}\left( x\right) \right) =x.  \label{lambda_invers}
\end{equation}%
So the mapping $F\left( x\right) \ni x\rightarrow w_{\lambda }\left(
x\right) \in F\left( x\right) $ can by extended to the isomorphism.

By \cite{ShirshovComm} the variety of the all commutative algebras is a
Shreier variety and by \cite{Lewin} all automorphisms of the free algebras
of these varieties are tame. So if $\Theta $ is the variety of the all
commutative algebras, then for $\lambda \in k^{\ast }$ we have that 
\begin{equation}
w_{\lambda }\left( x\right) =\varphi \left( \lambda \right) x,
\label{scalar_mult}
\end{equation}%
where $\varphi \left( \lambda \right) \in k$. If $\lambda =0$, then must
fulfills $w_{\lambda }\left( x\right) =0$, so in this case we also can write
(\ref{scalar_mult}), where $\varphi \left( \lambda \right) =0$.

If $\Theta $ is the variety of the all power associative algebras, then $%
F\left( x\right) $ is the algebra of the polynomials of degrees no less then 
$1$. Hence from (\ref{lambda_invers}) we can conclude that $\deg w_{\lambda
}\left( x\right) =1$ and (\ref{scalar_mult}) holds. Similar result we have
for the variety of the all alternative algebras and for the variety of the
all Jordan algebras, because these varieties are subvarieties of the variety
of the all power associative algebras (see \cite[Chapter 2, Theorem 2]%
{ZhSlSheShiAlmost} and \cite[Chapter 3, Corollary from Theorem 8]%
{ZhSlSheShiAlmost}), so in these varieties $F\left( x\right) $ is also the
algebra of the polynomials of degrees no less then $1$.

We conclude (\ref{scalar_mult}) for the arbitrary subvariety defined by
identities with coefficients from $%
\mathbb{Z}
$ of the variety of the all anticommutative algebras from the fact that in
this variety $\dim F\left( x\right) =1$. Therefore in all our varieties we
have (\ref{scalar_mult}) for all $\lambda \in k$.

$\lambda \ast \left( \mu \ast x\right) =\left( \lambda \mu \right) \ast x$
must fulfills in $F\left( x\right) $ for every $\lambda ,\mu \in k$. We can
conclude from this axiom as in \cite{Tsurkov} that $\varphi \left( \lambda
\mu \right) =\varphi \left( \lambda \right) \varphi \left( \mu \right) $.
Also by using \cite[Proposition 4.2]{TsurkovAutomEquiv} we can prove that $%
\varphi :k\rightarrow k$ is a surjection.

After this we can conclude from axioms $x_{1}\perp 0=x_{1}$, $0\perp
x_{2}=x_{2}$, $x_{1}\perp x_{2}=x_{2}\perp x_{1}$ and $\lambda \ast \left(
x_{1}\perp x_{2}\right) =\left( \lambda \ast x_{1}\right) \perp \left(
\lambda \ast x_{2}\right) $ as in the \cite{Tsurkov} that in all our
varieties the%
\begin{equation}
w_{+}\left( x_{1},x_{2}\right) =x_{1}+x_{2}  \label{addition}
\end{equation}%
holds. Hear we must use the decomposition of $F\left( x_{1},x_{2}\right) $
to the direct sum of the linear spaces of elements which are homogeneous
according the sum of degrees of generators, which was used in \cite{Tsurkov}.

From axiom $\left( \lambda +\mu \right) \ast x=\lambda \ast x+\mu \ast x$
for every $\lambda ,\mu \in k$ we conclude that $\varphi \left( \lambda +\mu
\right) =\varphi \left( \lambda \right) +\varphi \left( \mu \right) $. So $%
\varphi \in \mathrm{Aut}k$.

Now we must to find the all possible forms of the word $w_{\cdot }\in
F\left( x_{1},x_{2}\right) $. Hear as in \cite{Tsurkov} we use the
decomposition of $F\left( x_{1},x_{2}\right) $ to the direct sum of the
linear spaces of elements which are homogeneous according the degree of $%
x_{1}$, and after this according the degree of $x_{2}$. From axioms $0\times
x_{2}=x_{1}\times 0=0$ and $\lambda \ast \left( x_{1}\times x_{2}\right)
=\left( \lambda \ast x_{1}\right) \times x_{2}=x_{1}\times \left( \lambda
\ast x_{2}\right) $ for every $\lambda \in k$ we can conclude that $w_{\cdot
}\in F_{2}\left( x_{1},x_{2}\right) $, $w_{\cdot }\left( x_{1},0\right)
=w_{\cdot }\left( 0,x_{2}\right) =0$. It means that for the variety of the
all power associative algebras%
\begin{equation}
w_{\cdot }\left( x_{1},x_{2}\right) =\alpha _{1,2}x_{1}x_{2}+\alpha
_{2,1}x_{2}x_{1},  \label{power_assoc_mult}
\end{equation}%
where $\alpha _{1,2},\alpha _{2,1}\in k$. By condition Op2) in this variety
the multiplication in $F_{W}^{\ast }$ can not by commutative or
anticommutative, so $\alpha _{1,2}\neq \pm \alpha _{2,1}$.

For the varieties of the all commutative algebras, of the all Jordan
algebras and for the arbitrary subvariety defined by identities with
coefficients from $%
\mathbb{Z}
$ of the variety of the all anticommutative algebras we have that%
\begin{equation}
w_{\cdot }\left( x_{1},x_{2}\right) =\alpha _{1,2}x_{1}x_{2},
\label{commut_mult}
\end{equation}%
where $\alpha _{1,2}\neq 0$.

For the variety of the all alternative algebras we conclude from axiom $%
\left( x_{1}\times x_{1}\right) \times x_{2}=x_{1}\times \left( x_{1}\times
x_{2}\right) $ that%
\begin{equation}
w_{\cdot }\left( x_{1},x_{2}\right) =\alpha _{1,2}x_{1}x_{2}
\label{altern_mult_1}
\end{equation}%
or%
\begin{equation}
w_{\cdot }\left( x_{1},x_{2}\right) =\alpha _{2,1}x_{2}x_{1},
\label{altern_mult_2}
\end{equation}%
where $\alpha _{1,2},\alpha _{2,1}\neq 0$.

Now we will prove for all our varieties that the systems of words $W$
defined above fulfill condition Op2). First of all we will prove that if $%
H\in \Theta $ then $H_{W}^{\ast }\in \Theta $. It means that we will check
that in the $H_{W}^{\ast }$ the all axioms of the variety $\Theta $ hold.
All these checking can be made by direct calculations. For all our varieties
we must check only these axioms of linear algebra: $\left(
x_{1}+x_{2}\right) \times x_{3}=\left( x_{1}\times x_{3}\right) +\left(
x_{2}\times x_{3}\right) $ and $x_{1}\times \left( x_{2}+x_{3}\right)
=\left( x_{1}\times x_{2}\right) +\left( x_{1}\times x_{3}\right) $, because
other axioms are immediately concluded from the forms of the words of the
system $W$.

For the varieties of the all commutative algebras and of the all Jordan
algebras we must check the axiom $x_{1}\times x_{2}=x_{2}\times x_{1}$.

Also for the variety of the all Jordan algebras we also must check the axiom 
$\left( \left( x_{1}\times x_{1}\right) \times x_{2}\right) \times
x_{1}=\left( x_{1}\times x_{1}\right) \times \left( x_{2}\times x_{1}\right) 
$.

For the variety of the all alternative algebras we must check the axioms $%
\left( x_{1}\times x_{1}\right) \times x_{2}=x_{1}\times \left( x_{1}\times
x_{2}\right) $ and $x_{2}\times \left( x_{1}\times x_{1}\right) =\left(
x_{2}\times x_{1}\right) \times x_{1}$.

For the variety of the all power associative algebras we must check the
axioms (\ref{p_ass_id}) where the original multiplication changed by
operation $\times $. We consider a power associative algebra $H$ and the
free $1$-generated power associative algebra $F=F\left( x\right) $. We take
a monomial $u\in F$ such that coefficient of $u$ is $1$, and $\deg u=n$. If
in the monomial $u$ we change the original multiplication by operation $%
\times $ than we achieve $u^{\times }\in F_{W}^{\ast }$. It is easy to prove
by induction that $u^{\times }\left( h\right) =\left( \alpha _{1,2}+\alpha
_{2,1}\right) ^{n-1}u\left( h\right) $ holds for every $h\in H$. It finishes
the checking of the necessary axioms.

For the arbitrary subvariety defined by identities with coefficients from $%
\mathbb{Z}
$ of the variety of the all anticommutative algebras we must check the axiom 
$x_{1}\times x_{2}=-$ $1\ast \left( x_{2}\times x_{1}\right) $ and the
specific axioms of this subvariety. Our field $k$ is infinite so all axioms
of this subvariety can by presented in the homogeneous form: $%
\sum\limits_{i=1}^{m}\lambda _{i}u_{i}=0$, where $\lambda _{i}\in 
\mathbb{Z}
$, $u_{i}\in F\left( x_{1},\ldots ,x_{r}\right) $, $F\left( x_{1},\ldots
,x_{r}\right) \in \mathrm{Ob}\Theta ^{0}$, $u_{i}$ are monomials with
coefficients $1$, $\deg u_{i}=n$ for every $i$. $\lambda _{i}\in 
\mathbb{Z}
$ because all operations of the process of the homogenization of the
identities can by made with coefficients from $%
\mathbb{Q}
$. We must check that for every $H\in \Theta $ and every $h_{1},\ldots
,h_{r}\in H$ the $\sum\limits_{i=1}^{m}\lambda _{i}\ast u_{i}^{\times
}\left( h_{1},\ldots ,h_{r}\right) =0$ holds. As above by induction we can
prove that $u_{i}^{\times }\left( h_{1},\ldots ,h_{r}\right) =\alpha
_{1,2}^{n-1}u_{i}\left( h_{1},\ldots ,h_{r}\right) $, so $%
\sum\limits_{i=1}^{m}\lambda _{i}\ast u_{i}^{\times }\left( h_{1},\ldots
,h_{r}\right) =\alpha _{1,2}^{n-1}\sum\limits_{i=1}^{m}\lambda
_{i}u_{i}\left( h_{1},\ldots ,h_{r}\right) =0$.

After all these checking we can conclude that for every $F=F\left( X\right)
\in \mathrm{Ob}\Theta ^{0}$ there exists a homomorphism $\sigma
_{F}:F\rightarrow F_{W}^{\ast }$ such that $\sigma _{F}\mid _{X}=id_{X}$. As
in the \cite{Tsurkov} we can prove that $\sigma _{F}$ is an isomorphism, so
the systems of words defined above fulfill condition Op2). It completes the
calculation of the group $\mathfrak{S}$.

For calculation of the group $\mathfrak{Y}\cap \mathfrak{S}$ we can prove as
in the \cite{Tsurkov} that for the all considered varieties the strongly
stable automorphism $\Phi $ which corresponds to the defined above system of
words $W$ is inner if and only if $\alpha _{2,1}=0$ and $\varphi =id_{k}$.

So, as in the \cite{Tsurkov} we can prove that for the variety of the all
power associative algebras $\mathfrak{A/Y\cong }\left( U\left( k\mathbf{S}_{%
\mathbf{2}}\right) \mathfrak{/}U\left( k\left\{ e\right\} \right) \right) 
\mathfrak{\leftthreetimes }\mathrm{Aut}k$, where $\mathbf{S}_{\mathbf{2}}$
is the symmetric group of the set which has $2$ elements, $U\left( k\mathbf{S%
}_{\mathbf{2}}\right) $ is the group of all invertible elements of the group
algebra $k\mathbf{S}_{\mathbf{2}}$, $U\left( k\left\{ e\right\} \right) $ is
a group of all invertible elements of the subalgebra $k\left\{ e\right\} $,
every $\varphi \in \mathrm{Aut}k$ acts on the algebra $k\mathbf{S}_{\mathbf{2%
}}$ by natural way: $\varphi \left( ae+b\left( 12\right) \right) =\varphi
\left( a\right) e+\varphi \left( b\right) \left( 12\right) $. But there is
an isomorphism of groups%
\begin{equation*}
U\left( k\mathbf{S}_{\mathbf{2}}\right) \ni ae+b\left( 12\right) \rightarrow
\left( a+b,a-b\right) \in k^{\ast }\times k^{\ast }
\end{equation*}%
so there is isomorphism%
\begin{equation*}
U\left( kS_{2}\right) /U\left( k\left\{ e\right\} \right) =U\left(
kS_{2}\right) /k^{\ast }e\ni \left( ae+b\left( 12\right) \right) k^{\ast
}e\rightarrow \frac{a+b}{a-b}\in k^{\ast }.
\end{equation*}%
Hence we prove the

\begin{theorem}
\label{group_for_power_associative copy(1)}For variety of the all power
associative algebras 
\begin{equation*}
\mathfrak{A/Y\cong }k^{\ast }\mathfrak{\leftthreetimes }\mathrm{Aut}k
\end{equation*}%
holds.
\end{theorem}

By similar way we prove

\begin{theorem}
\label{group_for_alternative}For the variety of the all alternative algebras%
\begin{equation*}
\mathfrak{A/Y\cong }\mathbf{S}_{\mathbf{2}}\times \mathrm{Aut}k
\end{equation*}%
holds.
\end{theorem}

And for other considered varieties we achieve

\begin{theorem}
\label{group_for_commutative}For the variety of the all commutative algebras%
\begin{equation*}
\mathfrak{A/Y\cong }\mathrm{Aut}k
\end{equation*}%
holds.
\end{theorem}

\begin{theorem}
\label{groupe_for_Jordan}For the variety of the all Jordan algebras 
\begin{equation*}
\mathfrak{A/Y\cong }\mathrm{Aut}k
\end{equation*}%
holds.
\end{theorem}

\begin{theorem}
\label{group_for_anticommutative}For the arbitrary subvariety defined by
identities with coefficients from $%
\mathbb{Z}
$ of the variety of the all anticommutative algebras 
\begin{equation*}
\mathfrak{A/Y\cong }\mathrm{Aut}k
\end{equation*}%
holds.
\end{theorem}

The results of this Section and of the \cite{Tsurkov} can be summarized in
this table:

\begin{tabular}{|c|c|c|}
\hline
& Variety of the all & $\mathfrak{A/Y}$ \\ \hline
1 & linear algebras & $k^{\ast }\mathfrak{\leftthreetimes }\mathrm{Aut}k$ \\ 
\hline
2 & commutative algebras & $\mathrm{Aut}k$ \\ \hline
3 & power associative algebras & $k^{\ast }\mathfrak{\leftthreetimes }%
\mathrm{Aut}k$ \\ \hline
4 & alternative algebras & $\mathbf{S}_{\mathbf{2}}\times \mathrm{Aut}k$ \\ 
\hline
5 & Jordan algebras & $\mathrm{Aut}k$ \\ \hline
6 & 
\begin{tabular}{c}
arbitrary subvariety of anticommutative \\ 
algebras defined by identities with coefficients from $%
\mathbb{Z}
$%
\end{tabular}
& $\mathrm{Aut}k$ \\ \hline
\end{tabular}

\section{Examples.\label{examples}}

\setcounter{equation}{0}

We can see from the previous Section that in the all considered varieties of
the linear algebras the group $\mathfrak{A/Y}$ is generated by cosets which
are presented by no more than two kinds of the strongly stable automorphisms
of the category $\Theta ^{0}$. One kind of automorphisms is connected to the
changing of the multiplication by scalar and second one is connected to the
changing of the multiplication of the elements of the algebras. In this
Section we will present some examples of the pairs of linear algebras such
that the considered automorphism provides the automorphic equivalence of
these algebras but these algebras are not geometrically equivalent. These
examples will be presented for the all considered above varieties of
algebras and for both these kinds of the strongly stable automorphisms, when
they exist in the group $\mathfrak{A/Y}$.

Same time we will use this simple

\begin{proposition}
\label{T}If $F\left( X\right) $ is a free universal algebra of the variety $%
\Theta $ and $T$ is an arbitrary congruence in $F\left( X\right) \times
F\left( X\right) $, then $T$ is an $H$-closed congruence, where $H=F\left(
X\right) /T$.
\end{proposition}

\begin{proof}
$\left( T\right) _{H}^{\prime \prime }=\bigcap\limits_{\substack{ \varphi
\in \mathrm{Hom}\left( F\left( X\right) ,H\right) ,  \\ \ker \varphi
\supseteq T}}\ker \varphi $. There is the natural epimorphism $\tau :F\left(
X\right) \rightarrow F\left( X\right) /T$, such that $\ker \tau =T$. So $%
\left( T\right) _{H}^{\prime \prime }=T$.
\end{proof}

\begin{example}
\label{aut_1_3_4}.
\end{example}

In this example we at first will denote by $\Theta $ the variety of the all
linear algebras. We will consider a strongly stable automorphism $\Phi $ of
the category $\Theta ^{0}$ which corresponds to the system of words%
\begin{equation}
W=\left\{ w_{0}=0,w_{\lambda }\left( x\right) =\varphi \left( \lambda
\right) x\left( \lambda \in k\right) ,w_{+}\left( x_{1},x_{2}\right)
=x_{1}+x_{2},w_{\cdot }\left( x_{1},x_{2}\right) =x_{1}x_{2}\right\}
\label{list_with_aut}
\end{equation}%
where $\varphi \in \mathrm{Aut}k$. We assume about automorphism $\varphi $
that there is $\lambda \in k$ such $\varphi \left( \lambda \right) \neq
\lambda $ and $\varphi \left( \lambda \right) \neq \lambda ^{-1}$. These
conditions fulfill, for example, if $k=k_{0}\left( \theta _{1},\theta
_{2}\right) $ is a transcendental extension of degree $2$ of the same
subfield $k_{0}$ and $\varphi \left( \theta _{1}\right) =\theta _{2}$.

We will consider the free algebra of $\Theta $ with two free generators: $%
F=F\left( x_{1},x_{2}\right) $. We consider the ideal $T=\left\langle
t,F^{3}\right\rangle $, where $t=\lambda x_{1}x_{2}+x_{2}x_{1}$. We will
denote $H=F/T$. By \cite[Theorem 5.1]{TsurkovAutomEquiv} $H$ and $%
H_{W}^{\ast }$ are automorphically equivalent.

\begin{proposition}
\label{prop_aut_1_3_4}$H$ and $H_{W}^{\ast }$ are not geometrically
equivalent.
\end{proposition}

\begin{proof}
By Proposition \ref{T} $T$ is an $H$-closed ideal. If $H$ and $H_{W}^{\ast }$
are geometrically equivalent then $T$ is an $H_{W}^{\ast }$-closed ideal.
The system of words (\ref{list_with_aut}) is a subject of Op1) and Op2), so
there exists the isomorphism $s_{F}:F\rightarrow F_{W}^{\ast }$ such that $%
s_{F}\left( x_{i}\right) =x_{i}$, $i=1,2$. By \cite[Remark 5.1]%
{TsurkovAutomEquiv} $s_{F}\left( T\right) $ is an $H$-closed ideal. By the
method of \cite{Tsurkov} we can prove that $s_{F}\left( F^{3}\right) =F^{3}$%
, so $s_{F}\left( T\right) =\left\langle s_{F}\left( t\right)
,F^{3}\right\rangle $, $s_{F}\left( t\right) =\varphi \left( \lambda \right)
x_{1}x_{2}+x_{2}x_{1}$. $\left( s_{F}\left( T\right) \right) _{H}^{\prime
\prime }=\bigcap\limits_{_{\substack{ \psi \in \mathrm{Hom}\left( F,H\right)
,  \\ \ker \psi \supseteq s_{F}\left( T\right) }}}\ker \psi $. We denote by $%
\tau $ the natural epimorphism $F\rightarrow F/T=H$. By projective propriety
of the free algebras for every $\psi \in \mathrm{Hom}\left( F,H\right) $
there exists $\alpha \in \mathrm{End}F$ such that $\psi =$ $\tau \alpha $.
We can write that $\alpha \left( x_{i}\right) =\alpha _{1i}x_{1}+\alpha
_{2i}x_{2}+f_{i}$, where $i=1,2$, $f_{i}\in F^{2}$, $\alpha _{ji}\in k$.%
\begin{equation*}
\alpha \left( s_{F}\left( t\right) \right) \equiv \left( \varphi \left(
\lambda \right) +1\right) \alpha _{11}\alpha _{12}x_{1}^{2}+\left( \varphi
\left( \lambda \right) \alpha _{11}\alpha _{22}+\alpha _{12}\alpha
_{21}\right) x_{1}x_{2}+
\end{equation*}%
\begin{equation*}
\left( \varphi \left( \lambda \right) \alpha _{12}\alpha _{21}+\alpha
_{11}\alpha _{22}\right) x_{2}x_{1}+\left( \varphi \left( \lambda \right)
+1\right) \alpha _{21}\alpha _{22}x_{2}^{2}\left( \func{mod}F^{3}\right) .
\end{equation*}%
If $\ker \psi \supseteq s_{F}\left( T\right) $ where $\psi =$ $\tau \alpha $
then $\alpha \left( s_{F}\left( t\right) \right) \in T$ or%
\begin{equation*}
\left( \varphi \left( \lambda \right) +1\right) \alpha _{11}\alpha
_{12}x_{1}^{2}+\left( \varphi \left( \lambda \right) \alpha _{11}\alpha
_{22}+\alpha _{12}\alpha _{21}\right) x_{1}x_{2}+
\end{equation*}%
\begin{equation*}
\left( \varphi \left( \lambda \right) \alpha _{12}\alpha _{21}+\alpha
_{11}\alpha _{22}\right) x_{2}x_{1}+\left( \varphi \left( \lambda \right)
+1\right) \alpha _{21}\alpha _{22}x_{2}^{2}=\rho \left( \lambda
x_{1}x_{2}+x_{2}x_{1}\right) ,
\end{equation*}%
where $\rho \in k$. $\varphi \left( \lambda \right) \neq -1$, so $\alpha
_{11}\alpha _{12}=0$ and $\alpha _{21}\alpha _{22}=0$. These conditions
fulfill in one of the following four cases.

\textit{Case 1.} $\alpha _{11}=\alpha _{22}=0$. In this case $\alpha
_{12}\alpha _{21}=\rho \lambda $, $\varphi \left( \lambda \right) \alpha
_{12}\alpha _{21}=\rho $, so $\varphi \left( \lambda \right) \lambda \rho
=\rho $. Therefore, by our assumption about $\varphi $, $\rho =0$ and,
because $\varphi \left( \lambda \right) \neq 0$, $\alpha _{12}\alpha _{21}=0$%
. If $\alpha _{12}=0$ then $\alpha \left( x_{2}\right) \in F^{2}$ and $\ker
\psi \supset \mathrm{sp}_{k}\left\{ x_{1}x_{2},x_{2}x_{1}\right\} $. If $%
\alpha _{21}=0$ then $\alpha \left( x_{1}\right) \in F^{2}$ and also $\ker
\psi \supset \mathrm{sp}_{k}\left\{ x_{1}x_{2},x_{2}x_{1}\right\} $.

\textit{Case 2.} $\alpha _{12}=\alpha _{21}=0$. In this case $\varphi \left(
\lambda \right) \alpha _{11}\alpha _{22}=\rho \lambda $, $\alpha _{11}\alpha
_{22}=\rho $, so $\varphi \left( \lambda \right) \rho =\rho \lambda $.
Therefore, by our assumption about $\varphi $, $\rho =0$ and as above $\ker
\psi \supset \mathrm{sp}_{k}\left\{ x_{1}x_{2},x_{2}x_{1}\right\} $.

\textit{Case 3.} $\alpha _{11}=\alpha _{21}=0$. In this case $\alpha \left(
x_{1}\right) \in F^{2}$ and $\ker \psi \supset \mathrm{sp}_{k}\left\{
x_{1}x_{2},x_{2}x_{1}\right\} $.

\textit{Case 4.} $\alpha _{12}=\alpha _{22}=0$. In this case $\alpha \left(
x_{2}\right) \in F^{2}$ and $\ker \psi \supset \mathrm{sp}_{k}\left\{
x_{1}x_{2},x_{2}x_{1}\right\} $.

So in all these cases $\left( s_{F}\left( T\right) \right) _{H}^{\prime
\prime }$ $\supset \mathrm{sp}_{k}\left\{ x_{1}x_{2},x_{2}x_{1}\right\} $
and $s_{F}\left( T\right) \cap F_{2}\neq \left( s_{F}\left( T\right) \right)
_{H}^{\prime \prime }\cap F_{2}$. Therefore $s_{F}\left( T\right) \neq
\left( s_{F}\left( T\right) \right) _{H}^{\prime \prime }$ and $s_{F}\left(
T\right) $ is is not an $H$-closed ideal. This contradiction finishes the
proof.
\end{proof}

It is clear that this example is valid for the variety of the all power
associative algebras and for the variety of the all alternative algebras.
Indeed, if we consider the free algebra $F=F\left( x_{1},x_{2}\right) $ in
one of these varieties, ideal $T=\left\langle t,F^{3}\right\rangle \subset F$%
, the quotient algebra $H=F/T$ and the system of words $W$ as in (\ref%
{list_with_aut}) then, as in the previous calculations, we can prove that
algebras $H$ and $H_{W}^{\ast }$ are automorphically equivalent but are not
geometrically equivalent.

\begin{example}
\label{aut_2_5}
\end{example}

In this example we at first will denote by $\Theta $ the variety of the all
commutative algebras. We will consider a strongly stable automorphism $\Phi $
of the category $\Theta ^{0}$ which corresponds to the system of words (\ref%
{list_with_aut}). We assume about automorphism $\varphi $ that $\varphi \neq
id_{k}$. It means there is $\lambda \in k$ such that $\varphi \left( \lambda
\right) \neq \lambda $.

We will consider the free algebra of $\Theta $ with two free generators: $%
F=F\left( x_{1},x_{2}\right) $. We consider the ideal $T=\left\langle
t,F^{4}\right\rangle $, where $t=\lambda x_{1}\left( x_{1}x_{2}\right)
+x_{2}\left( x_{1}^{2}\right) $. We will denote $H=F/T$. Algebras $H$ and $%
H_{W}^{\ast }$ are automorphically equivalent.

\begin{proposition}
$H$ and $H_{W}^{\ast }$ are not geometrically equivalent.
\end{proposition}

\begin{proof}
As in the proof of the Proposition \ref{prop_aut_1_3_4} $s_{F}\left(
T\right) =\left\langle s_{F}\left( t\right) ,F^{4}\right\rangle $. But now $%
s_{F}\left( t\right) =\varphi \left( \lambda \right) x_{1}\left(
x_{1}x_{2}\right) +x_{2}\left( x_{1}^{2}\right) $. As in that proof we will
consider all endomorphisms $\alpha \in \mathrm{End}F$ such that $\alpha
\left( s_{F}\left( t\right) \right) \in T$ and will calculate the $\ker \psi 
$ where $\psi =\tau \alpha $, $\tau $ is the natural epimorphism $%
F\rightarrow F/T=H$. As above $\alpha \left( x_{i}\right) =\alpha
_{1i}x_{1}+\alpha _{2i}x_{2}+f_{i}$, where $i=1,2$, $f_{i}\in F^{2}$.%
\begin{equation*}
\alpha \left( s_{F}\left( t\right) \right) \equiv \left( \varphi \left(
\lambda \right) +1\right) \alpha _{11}^{2}\alpha _{12}x_{1}^{3}+\left(
\varphi \left( \lambda \right) \alpha _{11}^{2}\alpha _{22}+\left( \varphi
\left( \lambda \right) +2\right) \alpha _{11}\alpha _{12}\alpha _{21}\right)
x_{1}\left( x_{1}x_{2}\right) +
\end{equation*}%
\begin{equation*}
\left( \varphi \left( \lambda \right) \alpha _{11}\alpha _{21}\alpha
_{22}+\alpha _{12}\alpha _{21}^{2}\right) x_{1}\left( x_{2}^{2}\right)
+\left( \varphi \left( \lambda \right) \alpha _{11}\alpha _{12}\alpha
_{21}+\alpha _{11}^{2}\alpha _{22}\right) x_{2}\left( x_{1}^{2}\right) +
\end{equation*}%
\begin{equation*}
\left( \varphi \left( \lambda \right) \alpha _{12}\alpha _{21}^{2}+\left(
\varphi \left( \lambda \right) +2\right) \alpha _{11}\alpha _{21}\alpha
_{22}\right) x_{2}\left( x_{1}x_{2}\right) +\left( \varphi \left( \lambda
\right) +1\right) \alpha _{21}^{2}\alpha _{22}x_{2}^{3}\left( \func{mod}%
F^{4}\right) .
\end{equation*}%
So if $\alpha \left( s_{F}\left( t\right) \right) \in T$ then $\alpha
_{11}\alpha _{12}=0$ and $\alpha _{21}\alpha _{22}=0$. As above we must
consider the following four cases:

\textit{Case 1.} $\alpha _{11}=\alpha _{22}=0$. In this case we conclude
from $\alpha \left( s_{F}\left( t\right) \right) \in T$ that $\alpha
_{12}\alpha _{21}=0$ holds. If $\alpha _{12}=0$ then $\alpha \left(
x_{2}\right) \in F^{2}$ and 
\begin{equation}
\ker \psi \supset \mathrm{sp}_{k}\left\{ x_{1}\left( x_{1}x_{2}\right)
,x_{1}\left( x_{2}^{2}\right) ,x_{2}\left( x_{1}^{2}\right) ,x_{2}\left(
x_{1}x_{2}\right) \right\} .  \label{ker_2_5}
\end{equation}%
If $\alpha _{21}=0$ then $\alpha \left( x_{1}\right) \in F^{2}$ and also (%
\ref{ker_2_5}) holds.

\textit{Case 2.} $\alpha _{12}=\alpha _{21}=0$. In this case we conclude
from $\alpha \left( s_{F}\left( t\right) \right) \in T$ that $\varphi \left(
\lambda \right) \alpha _{11}^{2}\alpha _{22}=\rho \lambda $, $\alpha
_{11}^{2}\alpha _{22}=\rho $, where $\rho \in k$. $\varphi \left( \lambda
\right) \neq \lambda $, so $\rho =0$ and $\alpha _{11}\alpha _{22}=0$. So,
as above $\alpha \left( x_{1}\right) \in F^{2}$ or $\alpha \left(
x_{2}\right) \in F^{2}$ and (\ref{ker_2_5}) holds.

\textit{Case 3.} $\alpha _{11}=\alpha _{21}=0$. In this case $\alpha \left(
x_{1}\right) \in F^{2}$ and (\ref{ker_2_5}) holds.

\textit{Case 4.} $\alpha _{12}=\alpha _{22}=0$. In this case $\alpha \left(
x_{2}\right) \in F^{2}$ and (\ref{ker_2_5}) holds.

So in all these cases $\left( s_{F}\left( T\right) \right) _{H}^{\prime
\prime }$ $\supset \mathrm{sp}_{k}\left\{ x_{1}\left( x_{1}x_{2}\right)
,x_{1}\left( x_{2}^{2}\right) ,x_{2}\left( x_{1}^{2}\right) ,x_{2}\left(
x_{1}x_{2}\right) \right\} $ and $s_{F}\left( T\right) \cap F_{3}\neq \left(
s_{F}\left( T\right) \right) _{H}^{\prime \prime }\cap F_{3}$. Therefore $%
s_{F}\left( T\right) \neq \left( s_{F}\left( T\right) \right) _{H}^{\prime
\prime }$. As above the proof is finished.
\end{proof}

It is clear that this example is also valid for the variety of the all
Jordan algebras.

\begin{example}
\label{aut_6}
\end{example}

The Uroboros 1 program designed for symbolic computation in Lie algebras was
used for the finding of this example. We denote by $\Theta $ the variety of
the all Lie algebras. As in the both previous examples we will consider a
strongly stable automorphism $\Phi $ of the category $\Theta ^{0}$ which
corresponds to the system of words (\ref{list_with_aut}). We assume about
automorphism $\varphi $ that $\varphi \neq id_{k}$. It means there is $%
\lambda \in k$ such $\varphi \left( \lambda \right) \neq \lambda $.

$L=L\left( x_{1},x_{2}\right) $ is a free Lie algebra with two free
generators. The algebra $L/L^{6}$ has a basis%
\begin{equation*}
\{x_{1}=e_{1},x_{2}=e_{2},[x_{1},x_{2}]=e_{3},[x_{1},\left[ x_{1},x_{2}%
\right] ]=e_{4},[\left[ x_{1},x_{2}\right] ,x_{2}]=e_{5},
\end{equation*}%
\begin{equation*}
\lbrack x_{1},\left[ x_{1},\left[ x_{1},x_{2}\right] \right] ]=e_{6},[x_{1},%
\left[ \left[ x_{1},x_{2}\right] ,x_{2}\right] ]=e_{7},[\left[ \left[
x_{1},x_{2}\right] ,x_{2}\right] ,x_{2}]=e_{8},
\end{equation*}%
\begin{equation*}
\lbrack x_{1},\left[ x_{1},\left[ x_{1},\left[ x_{1},x_{2}\right] \right] %
\right] ]=e_{9},[x_{1},\left[ x_{1},\left[ \left[ x_{1},x_{2}\right] ,x_{2}%
\right] \right] ]=e_{10},[x_{1},\left[ \left[ \left[ x_{1},x_{2}\right]
,x_{2}\right] ,x_{2}\right] ]=e_{11},
\end{equation*}%
\begin{equation*}
\lbrack \left[ x_{1},\left[ x_{1},x_{2}\right] \right] ,\left[ x_{1},x_{2}%
\right] ]=e_{12},[\left[ x_{1},x_{2}\right] ,\left[ \left[ x_{1},x_{2}\right]
,x_{2}\right] ]=e_{13},[\left[ \left[ \left[ x_{1},x_{2}\right] ,x_{2}\right]
,x_{2}\right] ,x_{2}]=e_{14}\},
\end{equation*}%
where multiplication in $L$ we denote by Lie brackets. It will be more
punctual to write $e_{1}=x_{1}+L^{6}$ and so on, but we chose the shorter
form of the notation. This basis was found by program Uroboros 1. We
consider the ideal $T=\left\langle t,L^{6}\right\rangle $, where $t=\lambda
e_{10}+e_{12}$. As above we denote $H=L/T$. Algebras $H$ and $H_{W}^{\ast }$
are automorphically equivalent.

\begin{proposition}
$H$ and $H_{W}^{\ast }$ are not geometrically equivalent.
\end{proposition}

\begin{proof}
As in the proof of the Proposition \ref{prop_aut_1_3_4} $s_{F}\left(
T\right) =\left\langle s_{F}\left( t\right) ,L^{6}\right\rangle $. But now $%
s_{F}\left( t\right) =\varphi \left( \lambda \right) e_{10}+e_{12}$. As
above we will consider all endomorphisms $\alpha \in \mathrm{End}F$ such
that $\alpha \left( s_{F}\left( t\right) \right) \in T$ and will calculate
the $\ker \psi $ where $\psi =\tau \alpha $, $\tau $ is the natural
epimorphism $L\rightarrow L/T=H$. As above $\alpha \left( x_{i}\right)
=\alpha _{1,i}x_{1}+\alpha _{2,i}x_{2}+f_{i}$, where $i=1,2$, $f_{i}\in
L^{2} $. $\alpha \left( s_{F}\left( t\right) \right) \equiv
\sum\limits_{i=9}^{14}\alpha _{i}e_{i}\left( \func{mod}L^{6}\right) $, where%
\begin{equation*}
\alpha _{9}=-\varphi \left( \lambda \right) \alpha _{1,1}^{2}\alpha
_{1,2}\left( \alpha _{1,1}\alpha _{2,2}-\alpha _{1,2}\alpha _{2,1}\right) ,
\end{equation*}%
\begin{equation*}
\alpha _{10}=\varphi \left( \lambda \right) \alpha _{1,1}\left( \alpha
_{1,1}\alpha _{2,2}-\alpha _{1,2}\alpha _{2,1}\right) \left( \alpha
_{1,1}\alpha _{2,2}+2\alpha _{1,2}\alpha _{2,1}\right) ,
\end{equation*}%
\begin{equation*}
\alpha _{11}=-\varphi \left( \lambda \right) \alpha _{2,1}\left( \alpha
_{1,1}\alpha _{2,2}-\alpha _{1,2}\alpha _{2,1}\right) \left( 2\alpha
_{1,1}\alpha _{2,2}+\alpha _{1,2}\alpha _{2,1}\right) ,
\end{equation*}%
\begin{equation*}
\alpha _{12}=-\alpha _{1,1}\left( \alpha _{1,1}\alpha _{2,2}-\alpha
_{1,2}\alpha _{2,1}\right) \left( \varphi \left( \lambda \right) \alpha
_{1,2}\alpha _{2,1}-\alpha _{1,1}\alpha _{2,2}+\alpha _{1,2}\alpha
_{2,1}\right) ,
\end{equation*}%
\begin{equation*}
\alpha _{13}=\alpha _{2,1}\left( \alpha _{1,1}\alpha _{2,2}-\alpha
_{1,2}\alpha _{2,1}\right) \left( -\varphi \left( \lambda \right) \left(
\alpha _{1,2}\alpha _{2,1}+\alpha _{1,1}\alpha _{2,2}\right) +\left( \alpha
_{1,1}\alpha _{2,2}-\alpha _{1,2}\alpha _{2,1}\right) \right) ,
\end{equation*}%
\begin{equation*}
\alpha _{14}=\varphi \left( \lambda \right) \alpha _{2,1}^{2}\alpha
_{2,2}\left( \alpha _{1,1}\alpha _{2,2}-\alpha _{1,2}\alpha _{2,1}\right) .
\end{equation*}%
These coefficients also were found by program Uroboros 1. If $\alpha \left(
s_{F}\left( t\right) \right) \in T$ then $\alpha \left( s_{F}\left( t\right)
\right) \equiv \lambda \rho e_{10}+\rho e_{12}\left( \func{mod}L^{6}\right) $%
, where $\rho \in k$. So if $\alpha \left( s_{F}\left( t\right) \right) \in
T $ then $-\varphi \left( \lambda \right) \alpha _{1,1}^{2}\alpha
_{1,2}\left( \alpha _{1,1}\alpha _{2,2}-\alpha _{1,2}\alpha _{2,1}\right) =0$
and $\varphi \left( \lambda \right) \alpha _{2,1}^{2}\alpha _{2,2}\left(
\alpha _{1,1}\alpha _{2,2}-\alpha _{1,2}\alpha _{2,1}\right) $. If $\alpha
_{1,1}\alpha _{2,2}-\alpha _{1,2}\alpha _{2,1}=0$ then $\alpha \left(
x_{1}\right) \equiv \beta \alpha \left( x_{2}\right) \left( \func{mod}%
L^{2}\right) $, where $\beta \in k$. Hence $\alpha \left[ x_{1},x_{2}\right]
\in L^{3}$ and $\alpha \left( L^{5}\right) \subset L^{6}$. Therefore $\ker
\psi \supseteq \mathrm{sp}_{k}\left\{ e_{9},\ldots ,e_{14}\right\} $.

If $\alpha _{1,1}\alpha _{2,2}-\alpha _{1,2}\alpha _{2,1}\neq 0$ then from $%
\alpha \left( s_{F}\left( t\right) \right) \in T$ we conclude $\alpha
_{1,1}\alpha _{1,2}=0$, $\alpha _{2,1}\alpha _{2,2}=0$. So as above we must
consider the following four cases:

\textit{Case 1.} $\alpha _{11}=\alpha _{22}=0$. In this case we conclude
from $\alpha \left( s_{F}\left( t\right) \right) \in T$ by consideration of
the coefficient $\alpha _{11}$ that $\alpha _{1,2}\alpha _{2,1}=0$. As above
from this equation we conclude that $\alpha \left( x_{1}\right) \in L^{2}$
or $\alpha \left( x_{2}\right) \in L^{2}$. So $\alpha \left[ x_{1},x_{2}%
\right] \in L^{3}$ and as above $\ker \psi \supseteq \mathrm{sp}_{k}\left\{
e_{9},\ldots ,e_{14}\right\} $.

\textit{Case 2.} $\alpha _{12}=\alpha _{21}=0$. In this case we conclude
from $\alpha \left( s_{F}\left( t\right) \right) \in T$ by consideration of
the coefficient $\alpha _{10}$ that $\varphi \left( \lambda \right) \alpha
_{1,1}^{3}\alpha _{2,2}^{2}=\lambda \rho $ and by consideration of the
coefficient $\alpha _{12}$ that $\alpha _{1,1}^{3}\alpha _{2,2}^{2}=\rho $. $%
\varphi \left( \lambda \right) \neq \lambda $, so $\rho =0$ and $\alpha
_{1,1}\alpha _{2,2}=0$. As above from this equation we conclude that $\alpha
\left( x_{1}\right) \in L^{2}$ or $\alpha \left( x_{2}\right) \in L^{2}$.
Therefore $\ker \psi \supseteq \mathrm{sp}_{k}\left\{ e_{9},\ldots
,e_{14}\right\} $.

In the \textit{Case 3:} $\alpha _{11}=\alpha _{21}=0$ - and in the \textit{%
Case 4:} $\alpha _{12}=\alpha _{22}=0$ - also $\alpha \left( x_{1}\right)
\in L^{2}$ or $\alpha \left( x_{2}\right) \in L^{2}$. So $\ker \psi
\supseteq \mathrm{sp}_{k}\left\{ e_{9},\ldots ,e_{14}\right\} $.

Therefore in all these cases $\left( s_{F}\left( T\right) \right)
_{H}^{\prime \prime }$ $\supseteq L^{5}\supsetneqq s_{F}\left( T\right) $.
\end{proof}

\begin{example}
\label{s_1_3}
\end{example}

In \cite[Section 5]{Tsurkov} was given an example of the algebras in the
variety of the all linear algebras which are automorphically equivalent but
not\textbf{\ }geometrically equivalent. The variety of the all linear
algebras was denoted by $\Theta $. We considered the strongly stable
automorphism $\Phi $ of the category $\Theta ^{0}$ corresponding to the
system of words%
\begin{equation*}
W=\left\{ w_{0}=0,w_{\lambda }\left( x\right) =\lambda x\left( \lambda \in
k\right) ,w_{+}\left( x_{1},x_{2}\right) =x_{1}+x_{2},w_{\cdot }\left(
x_{1},x_{2}\right) =ax_{1}x_{2}+bx_{2}x_{1}\right\} ,
\end{equation*}%
where $b\neq 0$. In the free algebra $F\left( x_{1},x_{2}\right) \in \mathrm{%
Ob}\Theta ^{0}$ we considered the verbal ideal $I$ generated by identity $%
\left( x_{1}x_{1}\right) x_{2}=0$. By \cite[Theorem 5.1]{TsurkovAutomEquiv}
the algebras $H=F\left( x_{1},x_{2}\right) /I$ and $H_{W}^{\ast }$ are
automorphically equivalent. We proved in \cite[Proposition 5.1]{Tsurkov}
that the algebras $H$ and $H_{W}^{\ast }$ are not geometrically equivalent.

From this proof it is clear that this example is valid for the variety of
the all power associative algebras.

\begin{example}
\label{s_4}
\end{example}

In this example we denote by $\Theta $ the variety of the all alternative
algebras. We will consider a strongly stable automorphism $\Phi $ of the
category $\Theta ^{0}$ which corresponds to the system of words%
\begin{equation*}
W=\left\{ w_{0}=0,w_{\lambda }\left( x\right) =\lambda x\left( \lambda \in
k\right) ,w_{+}\left( x_{1},x_{2}\right) =x_{1}+x_{2},w_{\cdot }\left(
x_{1},x_{2}\right) =x_{2}x_{1}\right\} .
\end{equation*}%
We will consider the free algebra of $\Theta $ with two free generators: $%
F=F\left( x_{1},x_{2}\right) $. This is an associative algebra. In this
algebra we will consider the verbal ideal $I$ generated by identity $%
x_{1}x_{2}^{2}=0$. We will denote $H=F/I$. As above $H$ and $H_{W}^{\ast }$
are automorphically equivalent. We will prove the

\begin{proposition}
$H$ and $H_{W}^{\ast }$ are not geometrically equivalent.
\end{proposition}

\begin{proof}
In this proof we use the method of \cite[Proposition 5.1]{Tsurkov}. The
ideal $I=\left\langle \alpha \left( x_{1}x_{2}^{2}\right) \mid \alpha \in 
\mathrm{End}F\right\rangle $ will be the smallest $H$-closed set in $F$,
because $I=\left( 0\right) _{H}^{\prime \prime }$, where $0\in F$. If
algebras $H$ and $H_{W}^{\ast }$ are geometrically equivalent then the
structures of the $H$-closed sets and of the $H_{W}^{\ast }$-closed sets in $%
F$ coincide. Hence $I$ must be the smallest $H_{W}^{\ast }$-closed set in $F$%
.

By \cite[Remark 5.1]{TsurkovAutomEquiv} 
\begin{equation}
T\rightarrow s_{F}T  \label{closed_bijection}
\end{equation}%
is a bijection from the structure of the $H_{W}^{\ast }$-closed sets in $F$
to the structure of the $H$-closed sets in $F$. It is clear that the
bijection (\ref{closed_bijection}) preserves inclusions of sets. So it
transforms the smallest $H_{W}^{\ast }$-closed set to the smallest $H$%
-closed set, hence $I=s_{F}I$ must fulfills.

We will get more information about the subspace $\left( I+F^{4}\right)
/F^{4} $. If as above $\alpha \in \mathrm{End}F$ such that $\alpha \left(
x_{i}\right) =\alpha _{1i}x_{1}+\alpha _{2i}x_{2}+f_{i}$, where $i=1,2$, $%
f_{i}\in F^{2}$, $\alpha _{ji}\in k$, then%
\begin{equation*}
\alpha \left( x_{1}x_{2}^{2}\right) \equiv \alpha _{11}\alpha
_{12}^{2}x_{1}^{3}+\alpha _{11}\alpha _{12}\alpha _{22}\left(
x_{1}^{2}x_{2}+x_{1}x_{2}x_{1}\right) +\alpha _{11}\alpha
_{22}^{2}x_{1}x_{2}^{2}+
\end{equation*}%
\begin{equation*}
\alpha _{12}^{2}\alpha _{21}x_{2}x_{1}^{2}+\alpha _{12}\alpha _{21}\alpha
_{22}\left( x_{2}x_{1}x_{2}+x_{2}^{2}x_{1}\right) +\alpha _{21}\alpha
_{22}^{2}x_{2}^{3}\left( \func{mod}F^{4}\right) .
\end{equation*}%
So $\left( I+F^{4}\right) /F^{4}\subseteq \mathrm{sp}_{k}\left\{ \overline{%
x_{1}^{3}},\overline{x_{1}^{2}x_{2}}+\overline{x_{1}x_{2}x_{1}},\overline{%
x_{1}x_{2}^{2}},\overline{x_{2}x_{1}^{2}},\overline{x_{2}x_{1}x_{2}}+%
\overline{x_{2}^{2}x_{1}},\overline{x_{2}^{3}}\right\} =V$, where $\overline{%
x_{1}^{3}}=x_{1}^{3}+F^{4}$ and so on. $s_{F}\left( x_{1}x_{2}^{2}\right)
=x_{2}^{2}x_{1}\in s_{F}I$, but $\overline{x_{2}^{2}x_{1}}\notin V$ and $%
\left( s_{F}I+F^{4}\right) /F^{4}\neq \left( I+F^{4}\right) /F^{4}$, so $%
I\neq s_{F}I$. This contradiction finishes the proof.
\end{proof}

\section{Acknowledgements.}

I am grateful to J. Semionova (Technological Qualification Institute,
Israel) and E. Shtranvasser (Open University, Israel) who developed the
program Uroborus 1, which was used in the finding of Example \ref{aut_6}.

I am thankful to Prof. I. P. Shestakov who was very heedful to my research.

I acknowledge the support by FAPESP - Funda\c{c}\~{a}o de Amparo \`{a}
Pesquisa do Estado de S\~{a}o Paulo (Foundation for Support Research of the
State S\~{a}o Paulo), project No. 2010/50948-2.

\end{document}